\RequirePackage{amsmath}
\RequirePackage{amsfonts,verbatim,mathdots}
\RequirePackage{amssymb}
\RequirePackage{dsfont}

\RequirePackage{array,multirow,makecell}

\RequirePackage{amsthm}

\RequirePackage{mathrsfs}
\RequirePackage{stmaryrd}

\documentclass[runningheads]{llncs}
\usepackage{rotating}
\usepackage{ifthen}
\usepackage{booktabs}
\usepackage{xcolor}
\usepackage{subcaption}
\usepackage{graphicx}
\makeatletter

\newcommand{\Rmnum}[1]{\expandafter\@slowromancap\romannumeral #1@} 
\makeatother

\newcommand{\markupdraft}[2]{
    \ifthenelse{\equal{#1}{display}}{#2}{}
    \ifthenelse{\equal{#1}{color}}{\color{#2}}{}
}

\newcommand{\newcolored}[3][]{{\markupdraft{color}{#2}#3}
    \ifthenelse{\equal{#1}{}}{}{\markupdraft{display}{{\color{yellow!70!black}[#1]}}}} 

\newcommand{\del}[2][]{{\markupdraft{display}{{\color{orange}[removed: "#2"[#1]]}}}} 

\newcommand{\indraftonly}[1]{{#1}}  
\renewcommand{\indraftonly}[1]{}\renewcommand{\markupdraft}[2]{}  
\renewcommand{\del}[2][]{}

\newcommand{\cheikh}[1]{\indraftonly{\bf {\color{purple}Cheikh: #1}}}

\newcommand{\ba}{\begin{eqnarray}}
\newcommand{\ea}{\end{eqnarray}}
\newcommand{\baStar}{\begin{eqnarray*}}
\newcommand{\eaStar}{\end{eqnarray*}}

\newcommand{\scal}[2]{\left<#1\,,\,#2\right>}

     \newcommand{\foncfast}[4]{#1 \ni #2 \longmapsto #3 \in #4}
  \newcommand{\fonccourte}[2]{#1 \longmapsto #2}

    \newcommand{\sbar}{\rule[-1pt]{0.4pt}{9pt}}   
\newcommand{\dsbar}{\sbar\,\sbar}       
    \newcommand{\norm}[1]{\,\dsbar\,#1\,\dsbar\,}

\newcommand{\dsp}{\displaystyle}
 
\newcommand{\pare}[1]{\left(#1\right)}
\newcommand{\croc}[1]{\left[#1\right]}

\DeclareMathOperator{\mydef}{def}

\newcommand{\egaldef}{\stackrel{{\scriptscriptstyle{\mydef}}}{=}}

\def\R{{\mathbb R}}
\def\diff{{\mathrm d}}

\DeclareMathOperator{\Img}{Im}

\begin{document}
\title{On Bi-Objective convex-quadratic problems}
%
%


\author{Cheikh Toure\inst{1} \and
Anne Auger\inst{1} \and
Dimo Brockhoff\inst{1} \and
Nikolaus Hansen\inst{1}}
\authorrunning{C. Toure et al.}
%
\institute{Inria and CMAP, Ecole Polytechnique, France\\firstname.lastname@inria.fr\\cheikh.toure@polytechnique.edu}
%
\maketitle              
\begin{abstract}
In this paper we analyze theoretical properties of bi-objective convex-quadratic problems. We give a complete description of their Pareto set and prove the convexity of their Pareto front. We show that the Pareto set is a line segment when both Hessian matrices are proportional.

We then propose a novel set of convex-quadratic test problems, describe their theoretical properties and the algorithm abilities required by those test problems. This includes in particular testing the sensitivity with respect to separability, ill-conditioned problems, rotational invariance, and whether the Pareto set is aligned with the coordinate axis.

\end{abstract}

\section{Introduction}
Convex-quadratic functions are among the simplest yet very useful test functions in optimization. Given a positive definite matrix $Q$ of $\mathbb{R}^{n \times n}$, a convex quadratic function is defined as 
$$
f(x) = \frac12 (x - x^*)^\top Q (x-x^*) \enspace
$$
where $x^*$ is the unique optimum of the function. The Hessian of $f$ coincides with the matrix $Q$. The level-sets of $f$ defined as
$
\{ x \in \R^n : (x - x^*)^\top Q (x-x^*) = c , c \geq 0 \}
$
are hyper-ellipsoids whose main axes are the eigenvectors of the matrix $Q$ with length proportional to the inverse of the eigenvalues of $Q$. 

By changing the eigenvalues and eigenvectors of $Q$, one can model different essential difficulties in numerical optimization: if the eigenvectors are not aligned with the coordinate axes (if the matrix $Q$ is not diagonal), then the associated function is non-separable: it cannot be efficiently optimized by coordinate-wise search. In practice, difficult optimization problems are non-separable. Having a large condition number for $Q$, that is a large ratio between the largest and smallest eigenvalue of $Q$ models ill-conditioned problems where the characteristic scale along different directions is very different. Ill-conditioning is very frequent in real-world problems. They arise naturally as one often optimizes quantities that have different natures and different intrinsic scales (some variables can be akin to time, others to weights, ...) such that a unit change along each variable can have a completely different impact on the function optimized. More generally, the eigenspectrum of $Q$ entirely characterizes the scale among the different axes of the hyper-ellipsoidal level sets and parametrizes the difficulty of the function: from the arguably easiest function, the sphere function $f(x) = \sum_{i=1}^n x_i^2$, to very difficult ill-conditioned functions where condition numbers of $Q$ of up to $10^{10}$ have been observed in real-world problems, for example in \cite{cdhq2010a}.

Convex-quadratic functions have been central to the design of several important classes of optimization algorithms for single-objective optimization. Newton or quasi-Newton methods use or learn a second order approximation of the objective function optimized \cite{nocedal2006numerical}. This second order approximation is done by convex-quadratic functions (assuming that the function is twice continuously differentiable and convex). Introduced more recently, the class of derivative-free-optimization (DFO) trust-region based algorithms builds a second-order approximation of the objective function by interpolation \cite{newuoaReport:2004}. In the evolutionary computation (EC) context, convex-quadratic functions have also played a central role for the design of algorithms like CMA-ES: they have been intensively used for designing the algorithm and the performance of the method has been carefully quantified on different eigenspectra of the matrix $Q$ for different condition numbers \cite{hansen2001}.

Given that a multiobjective problem is ``simply" the simultaneous optimization of single-objective problems, the typical difficulties of each objective function are the same as the typical difficulties of single-objective problems. In particular non-separability and ill-conditioning are important difficulties that the single functions have. Therefore, combining convex-quadratic problems seems natural for testing and designing multiobjective algorithms. This has already been done in the past for instance for the design of multiobjective versions of CMA-ES \cite{ihr2007a} or as a subset of the biobjective BBOB test function suite \cite{bth2015a,tbha2016a}.

Yet, while the difficulties encoded and parametrized within a convex-quadratic problem are well-understood for single-objective optimization, the situation is different for multiobjective optimization, starting from bi-objective optimization. Simple properties like convexity of the Pareto front  associated to bi-objective convex-quadratic problems as well as properties of the Pareto set have not been systematically investigated. Additionally, convex-quadratic bi-objective test problems used in the literature do not capture all important properties one could be testing with convex-quadratic problems. There is  more degree of freedom than for single objective optimization that is not exploited: we can combine two functions having the same Hessian matrix, place the optima on the functions both on one axis of the search space, ... and this will affect how the Pareto set and Pareto front look like.

This paper aims at filling the gaps from the literature on multiobjective optimization with respect to convex-quadratic problems. More precisely the objectives are twofold:
clarify theoretical Pareto properties of bi-objective problems where each function is convex-quadratic and define sets of bi-objective convex-quadratic problems that allow to test different (well-understood) difficulties of bi-objective problems. The paper is organized as follows: in Section~\ref{paretosection} we present theoretical properties of convex-quadratic problems and discuss new test functions in Section~\ref{classsection}.

\section{Theoretical Properties of Bi-Objective Convex-Quadratic Problems}
\label{paretosection}
\subsection{Preliminaries}

We consider bi-objective problems $(f_1,f_2)$ defined on the search space $\R^n$.

The Pareto set of $(f_1,f_2)$ is defined as the set of all non-dominated (or efficient) solutions $\{ x\in\R^n \,|\, \not\!\!\exists y\in\R^n \text{ such that } f_1(y)\leq f_1(x) \text{ and } f_2(y)\leq f_2(x) \text{ and at least one inequality is strict} \}$. The image of the Pareto set (in the objective space $\R^2$) is called the Pareto front of $(f_1,f_2)$. We first remark that the Pareto set remains unchanged if we compose the objective functions with a strictly increasing function. More precisely the following lemma holds.
\begin{lemma}[Invariance of the Pareto set to strictly increasing transformations of the objectives]
Given a bi-objective problem $x \mapsto (f_1(x), f_2(x))$ and $g_{1}: \fonccourte{\Img({f}_{1})}{\R}$, $g_{2}: \fonccourte{\Img({f}_{2})}{\R}$ two strictly increasing functions, then $({f}_{1},{f}_{2})$ and $(g_{1}\circ {f}_{1},g_{2}\circ {f}_{2})$ have the same Pareto set.  
\label{lemmaInvarianceIncreasing}
\end{lemma}
\begin{proof}
If $x$ is not in the Pareto set of $(g_{1}\circ {f}_{1},g_{2}\circ {f}_{2})$, then their exists $y$ such that $g_{1}\circ {f}_{1}(y) \leq g_{1}\circ {f}_{1}(x)$ and $g_{2}\circ {f}_{2}(y) \leq g_{2}\circ {f}_{2}(x)$ with one inequality being strict, which is equivalent to the fact that  ${f}_{1}(y) \leq {f}_{1}(x)$ and $f_{2}(y) \leq {f}_{2}(x)$, with one inequality being strict. And vice versa. Hence $x$ is not in the Pareto set of  $(g_{1}\circ {f}_{1},g_{2}\circ {f}_{2})$ if and only if it is not in the Pareto set of $(f_1,f_2)$, which shows that both problems have the same Pareto set.
\qed
\end{proof}
From now on $(f_1,f_2)$ denote a bi-objective convex-quadratic problem. More precisely, let $x_{1}$, $x_{2}$ be two \emph{different} vectors in $\R^{n}$, and $\alpha,\beta > 0$. Let  $Q_{1}$ and $Q_{2}$ (in $\R^{n^2}$) be two positive definite matrices and  consider the bi-objective \emph{minimization} problem $ \pare{f_{1},f_{2}}$ defined for $x \in \R^n$ as
\begin{align}
f_1 (x) =  \frac{1}{\alpha} \pare{x-x_{1}}^\top Q_{1} \pare{x-x_{1}},
f_2 (x) =  \frac{1}{\beta} \pare{x-x_{2}}^\top Q_{2} \pare{x-x_{2}}.
\end{align}
We denote this general bi-objective convex-quadratic problem by $\mathcal{P}$, and assume that the optimization goal is to find (an approximation of) the Pareto set of $\mathcal{P}$.
\subsection{Pareto set}
We characterize in this section the Pareto set of $\mathcal{P}$.
We use the linear scalarization method to obtain the whole Pareto set. This is doable, whenever $f_{1}$ and $f_{2}$ are strict convex functions (see~\cite{jahnvector}). Then the Pareto set of $\mathcal{P}$ is described by the solutions of 
$$\dsp\min_{x\in\R^{n}} \, \pare{1-t} f_{1}(x)+t f_{2}(x) \text{ , for }  t \in \croc{0,1}. $$

We prove in the next proposition that the Pareto set of $\mathcal{P}$ is a continuous and differentiable  parametric curve of $\R^n$ whose extremes are $x_1$ and $x_2$.
\begin{proposition}
 The Pareto set of $\mathcal{P}$ is the image of the function $\varphi$ defined as 
 \begin{equation}
 \varphi: t \in [0,1] \mapsto \croc{ (1-t) Q_{1} + t Q_{2} }^{-1} \croc{ (1-t)Q_{1}x_{1}+t Q_{2}x_{2} }   \enspace.
 \end{equation}
The function $\varphi$ is differentiable and verifies for any $t$ in $ \croc{0,1} $
\begin{align}\label{eq:ParetoSet}
\dsp (1-t)Q_{1}\pare{ \varphi(t)-x_{1} } & = tQ_{2}\pare{ x_{2} - \varphi(t) }, \\ \label{derivephi}
 t\croc{ (1-t) Q_{1} + t Q_{2} } \varphi^{\prime}(t) & = Q_{1}\pare{ \varphi(t)-x_{1} } .
\end{align}
Hence, the Pareto set is a continuous (differentiable) curve of $\R^n$ whose extremes are $x_1=\varphi(0)$ and $x_2=\varphi(1)$.
\label{paretoset}
\end{proposition}

\begin{proof}
For any $s$ in $\croc{0,1}$, define $ g_{s} \egaldef (1-s)f_{1}+s f_{2}.$ We observe that $g_{s}$, like $f_{1}$ and $f_{2}$, is strictly convex, differentiable, and diverges to $\infty$ when $\|x \|$ goes to $\infty$ (where $\| x \|$ denotes the Euclidean norm). Then its critical point minimizes $g_{s}.$
 Let us now compute the gradient of $g_{s}$ times $\alpha \beta$ for $x$ in $\R^{n}$:
 $$ \dsp\alpha\beta \nabla g_{s}(x) = \dsp (1-s)\alpha\beta\, \nabla f_{1}(x) + s\alpha\beta\, \nabla f_{2}(x)
    = \dsp 2(1-s)\beta\, Q_{1} (x-x_{1}) + 2s\alpha\, Q_{2}(x - x_{2})$$
     $$\text{Thus, }\dsp\alpha\beta \nabla g_{s}(x) = \dsp 2 \croc{ (1-s)\beta\, Q_{1} + s\alpha\, Q_{2} } x - 2(1-s)\beta\, Q_{1}x_{1} - 2s\alpha\, Q_{2} x_{2}.$$
Then it follows that for any $s$ in $\croc{0,1}$, the point that minimizes $g_{s}$ (its critical point), denoted by $\tilde{x}_{s}$ verifies
$\dsp  \frac{ (1-s)\beta\, Q_{1} + s\alpha\, Q_{2} }{(1-s)\beta+s\alpha} \,\tilde{x}_{s} = \frac{(1-s)\beta\, Q_{1}x_{1} +s\alpha\, Q_{2} x_{2}}{(1-s)\beta+s\alpha}$.
Since $\dsp\foncfast{ \croc{0,1} }{s}{ \frac{s\alpha}{(1-s)\beta+s\alpha} }{ \croc{0,1} }$ is bijective (its derivative is $\dsp\fonccourte{s}{\frac{\alpha\beta}{\pare{(1-s)\beta+s\alpha}^{2}}}$), then it is equivalent to parametrize the Pareto set with $\dsp t \egaldef  \frac{s\alpha}{(1-s)\beta+s\alpha}$. Hence, the Pareto set is fully described by $\pare{\varphi(t)}_{t\in\croc{0,1}}$ such that:

\ba
 \croc{ (1-t) Q_{1} + t Q_{2} } \varphi(t) &=& (1-t) Q_{1}x_{1} + t Q_{2} x_{2}, \label{phirelation2}\\
 (1-t)Q_{1}\pare{ \varphi(t)-x_{1} } &=& tQ_{2}\pare{ x_{2} - \varphi(t) } \label{phirelation1}.
\ea
The function $t \to \croc{ (1-t) Q_{1} + t Q_{2} }^{-1}$ is differentiable as inverse of a differentiable and invertible matrix function. Then $\varphi$ is differentiable. \\We differentiate~\eqref{phirelation2} and multiply by $t$ to obtain
$ t \croc{ (1-t) Q_{1} + t Q_{2} } \varphi^{\prime}(t) = tQ_{2}x_{2}-tQ_{1}x_{1} + t Q_{1}\varphi(t) - t Q_{2}\varphi(t) .$
Injecting in \eqref{phirelation1} gives
$ t\croc{ (1-t) Q_{1} + t Q_{2} } \varphi^{\prime}(t) = Q_{1}\pare{ \varphi(t)-x_{1} }, \text{ for any $t \in \croc{0,1}$}. $\qed
\end{proof}

We obtain as corollary that when $f_{1}$ and $f_{2}$ have proportional Hessian matrices, then the Pareto set is the line segment between the optima of the functions $f_1$ and $f_2$. 
\begin{corollary}
In the case where $f_{1}$ and $f_{2}$ have proportional Hessian matrices, the Pareto set of $\mathcal{P}$ is the line segment between $x_{1}$ and $x_{2}$.
\label{segmentset}
\end{corollary}
\begin{proof}
In that case, their exists a real $\gamma$ such that $\dsp\frac{Q_{1}}{\alpha}=\dsp\gamma\frac{Q_{2}}{\beta}.$ Then, Proposition~\ref{paretoset} implies that for any $t \in \croc{0, 1}$,
$$ \varphi(t) = \croc{ (1-t) \gamma\frac{\alpha}{\beta}Q_{2} + t Q_{2} }^{-1} \croc{ (1-t)\gamma\frac{\alpha}{\beta}Q_{2} x_{1}+t Q_{2} x_{2} } = \frac{\gamma\alpha(1-t)x_{1} +t\beta x_{2} }{(1-t)\alpha\gamma + t\beta},
$$
which is $\croc{x_{1},x_{2}}$, since $\foncfast{ \croc{0,1} }{t}{ \frac{t\beta }{(1-t)\alpha\gamma + t\beta}
  }{ \croc{0,1} }$ is a bijection. 
\label{paretosetSamequadratics}\qed
\end{proof}
Using Lemma~\ref{lemmaInvarianceIncreasing}, we directly deduce the following corollary.
\begin{corollary}
If $f_{1}$ and $f_{2}$ have proportional Hessian matrices, $g_{1}: \fonccourte{\Img(f_{1})}{\R}$, $g_{2}: \fonccourte{\Img(f_{2})}{\R}$ are two strictly increasing functions, then the Pareto set of the problem $\pare{g_{1}\circ f_{1}, g_{2}\circ f_{2} }$ is the line segment between $x_{1}$ and $x_{2}$.
\label{increasingTransformation}
\end{corollary}
As an example, the double-norm problem defined as: \\$\pare{ x\to\norm{x-x_{1}}_{2}, x\to\norm{x-x_{2}}_{2}}$ can be seen as: $(g\circ f_{1}, g\circ f_{2})$ where $g(x) = \sqrt{x}$, $f_{1}(x) = \norm{x-x_{1}}_{2}^{2} $ and $f_{2}(x) = \norm{x-x_{2}}_{2}^{2}$.\\
Then $( g\circ f_{1}, g\circ f_{2} )$ has the same Pareto set than the double-sphere problem $(f_{1}, f_{2})$, which is the line segment between $x_{1}$ and $x_{2}$. Therefore the Pareto front of the double-norm problem is described by $\pare{ t \norm{x_{2}-x_{1}}_{2}, (1-t) \norm{x_{2}-x_{1}}_{2}  }_{t\in\croc{0,1}} $. Thereby, the front is described by the function $\fonccourte{u}{\norm{x_{2}-x_{1}}_{2}-u.}$ We recover the well-known result that the double-norm problem has a linear front.

Corollary~\ref{increasingTransformation} allows also to recover the Pareto set description for the one-peak scenario in the Mixed-Peak Bi-Objective Problem (see \cite{kerschke2018search} and \cite{kerschke2016towards}).\footnote{In that scenario, we set $f_1 (x) =  \pare{x-c}^\top \Sigma \pare{x-c}$,
$f_2 (x) =   \pare{x-c^{\prime} }^\top \Sigma^{\prime} \pare{x-c^{\prime}}$ ($f_{1}$ and $f_{2}$
are seen as squares of the Mahalanobis distance to the optima, with respect to the Hessian matrices), $\dsp g_{1}(u) = 1-\frac{h_{1}}{ 1+\frac{\sqrt{u}}{r_{1}} }$, $\dsp g_{2}(u) = 1-\frac{h_{2}}{ 1+\frac{\sqrt{u}}{r_{2}} }$.}

In general, the Pareto set of a bi-objective convex-quadratic problem is not necessarily a line segment. Consider for instance for $n=2$ the case where $x_{1} = (0,0)^\top $, $x_{2} = (1,1)^\top $ and where we generate two different matrices $Q_1$ and $Q_2$ by randomly rotating a diagonal matrix with eigenvalues $1$ and $10$. Two resulting Pareto fronts associated to different random rotations are depicted in Figure~\ref{generalset}.

For $n=10$, we  also define $\mathcal{P}_{10}$ setting $x_{1} = (0,\dots,0)^\top $, $x_{2} = (1,\dots,1)^\top $ and $Q_{1}$ and $Q_{2}$ as diagonal matrices such that for $i=1,\ldots, 10$
 \begin{equation}\label{eq:Qdim10}
Q_{1}(i,i) = 100^{\frac{i-1}{9}}, \text{and } Q_{2}(i,i) = 10^{\frac{i-1}{9}} .
 \end{equation}
The different coordinates of the Pareto set given in \eqref{eq:ParetoSet} are depicted in Figure~\ref{generalset}.

\begin{figure*}[h!]
        \includegraphics[width=0.52\textwidth]{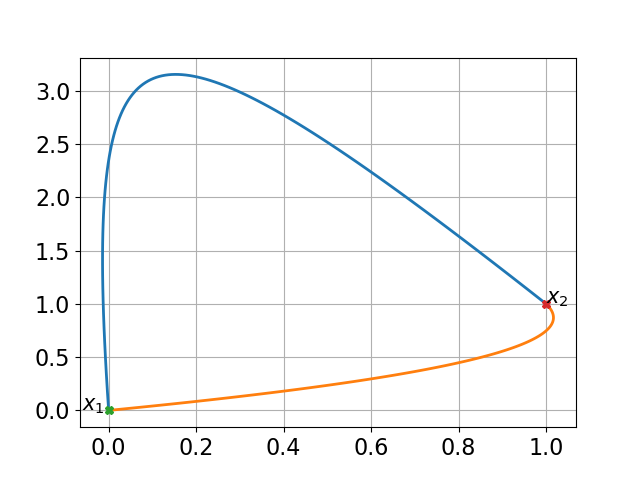}
 \hspace{-0.5cm}
        \includegraphics[width=0.52\textwidth]{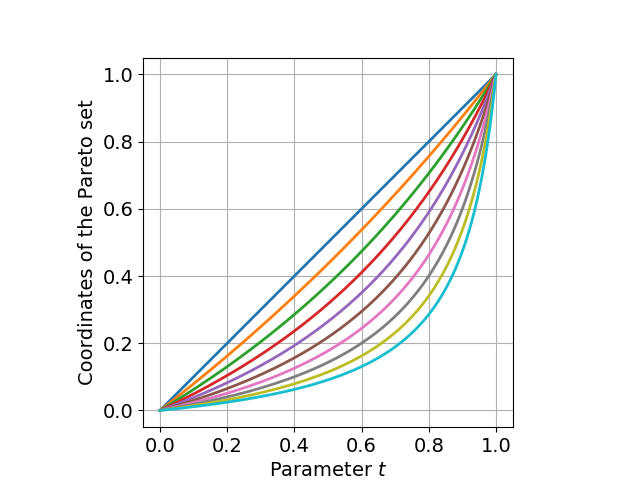}
    \caption{ Left: Two Pareto sets for $n=2$ represented in $\R^2$ with $Q_1$ and $Q_2$ randomly sampled and different. Right: Pareto set for $n=10$ with matrices given in \eqref{eq:Qdim10} represented as the function of the parameter $t$ given in \eqref{eq:ParetoSet}. The coordinates are ordered, the first one is on top and last one below. 
    }
    \label{generalset}
\end{figure*}

\subsection{Convexity of the Pareto front}
Corollary \ref{segmentset} proves that in the case where we have proportional Hessian matrices in  problem $\mathcal{P}$, the Pareto set is a line segment. Then it is reasonable to expect a simple analytic expression for the corresponding Pareto front. 
In what follows, we will express the Pareto front of a bi-objective problem as a one-dimensional function $u \in \R \mapsto g(u)$. Formally, if $t \in \R \mapsto \varphi(t) \in \R^n$ is a parametrization of the Pareto set, then the function $g$ satisfies $f_2(\varphi(t)) = g( f_1(\varphi(t))$. It is well-known that when $(f_1,f_2)$ is the double-sphere, that is $f_1(x) = \frac1n \sum_{i=1}^n x_i^2$ and $f_2(x) = \frac1n\sum_{i=1}^n (x_i - 1)^2$, then the Pareto front expression is given by $g(u) = (1-\sqrt{u})^2$ ~\cite{emmerich2007test}. In the next proposition, we show that this expression of the Pareto front holds (up to a normalization) for all bi-objective convex-quadratic problems, provided the Hessians of $f_1$ and $f_2$ are proportional. 
\begin{proposition}
When we have proportional Hessian matrices in the problem $\mathcal{P}$, the Pareto front is described by the following continuous and convex function:
\begin{equation}\label{paretofront}
u \in [0, \kappa_{\alpha}] \mapsto  \kappa_{\beta} \pare{ 1- \dsp\sqrt{ \frac{ u}{ \kappa_{\alpha} } } }^{2},
\text{where $\left\{\begin{array}{c} \kappa_{\alpha} = \dsp \frac{\pare{x_{2}-x_{1}}^\top Q_{1} \pare{x_{2}-x_{1}}}{\alpha}\\ \kappa_{\beta} = \dsp \frac{\pare{x_{2}-x_{1}}^\top Q_{2} \pare{x_{2}-x_{1}}}{\beta} \end{array}\right.$}
\end{equation}
\label{propHessianFront}
\end{proposition}
\begin{proof}
Denote $u \egaldef f_{1}\circ \varphi \text{ and } v \egaldef f_{2}\circ \varphi$,
  where $\varphi:\dsp\foncfast{\croc{0,1}}{t}{(1-t) x_{1}+t x_{2}}{\croc{x_{1},x_{2}} } $ is the line segment between $x_{1}$ and $x_{2}$.\\
For any $t\in\croc{0,1}$,
$u(t) = f_{1}(\varphi(t)) =  \frac{1}{\alpha} \pare{x_{2}-x_{1}}^\top Q_{1} \pare{x_{2}-x_{1}} t^{2},
 v(t) = f_{2}(\varphi(t)) =  \frac{1}{\beta} \pare{x_{2}-x_{1}}^\top Q_{2} \pare{x_{2}-x_{1}} \pare{1-t}^{2}.
 $
 It follows that for any $t\in\croc{0,1}$:
 $$\dsp v(t) = \dsp \frac{\pare{x_{2}-x_{1}}^\top Q_{2} \pare{x_{2}-x_{1}}}{\beta} \pare{ 1- \dsp\sqrt{ \frac{\alpha u(t)}{ \pare{x_{2}-x_{1}}^\top Q_{1} \pare{x_{2}-x_{1}} } } }^{2}.\,\,\qed$$
\end{proof}
From Proposition~\ref{propHessianFront}, we deduce that if we set 
 $ \kappa_{\alpha} = \kappa_{\beta} = 1$, then the  Pareto front will be independent from the Hessian matrix and will be described by the front of the double-sphere problem: $u \mapsto  ( 1- \sqrt{u})^{2}$.

We investigate now the general case where the Hessians of the functions $f_1$ and $f_2$ are not necessarily proportional.
Yet, before digging into the general convex-quadratic problems, we show a result on the shape of the Pareto front of a larger class of bi-objective problems.
\begin{theorem}
Let $f_{1}: \fonccourte{\R^{n}}{\R}$ and $f_{2}: \fonccourte{\R^{n}}{\R}$ be strict convex differentiable functions such that the problem $\pare{f_{1},f_{2}}$ has, as Pareto set, the image of a differentiable function $\varphi : \fonccourte{ \croc{0,1} }{\R^{n}}$.\\
Assume that: (i) $f_{1}\circ \varphi$ is strictly monotone, (ii) $\dsp\lim_{t\to 0} \frac{ (f_{1}\circ \varphi)^{\prime} (t) }{t}\neq 0$ and (iii) $\dsp\lim_{t\to 1} \frac{ (f_{2}\circ \varphi)^{\prime} (t) }{1-t}\neq 0$. Then, the Pareto front is a \textbf{convex} curve, with \textbf{vertical} tangent at $t = 0$ and \text{horizontal} tangent at $t = 1$.
 \label{bigtheo}
\end{theorem}

\begin{proof}
Denote by $u \egaldef f_{1}\circ \varphi \text{ and } v \egaldef f_{2}\circ \varphi.$
  Then the Pareto front is described by the parametric equation $ \pare{u(t),v(t)}, \text{for $t \in \croc{0,1}$.}$
We will show that $  u^{\prime} v^{\prime\prime} - u^{\prime\prime} v^{\prime} > 0 $ which implies the convexity of the curve.

By linear scalarization (see~\cite{jahnvector}, or weighted sum method in~\cite{grodzevich2006normalization}), as in the proof of Proposition~\ref{paretoset}, we have 
$ \dsp (1-t) \nabla f_{1}( \varphi(t) ) + t \nabla f_{2}( \varphi(t) ) = 0.$
If we take the scalar product of the former equation with $\varphi^{\prime}(t)$, we obtain that
\ba \dsp (1-t) \scal{\nabla f_{1}( \varphi(t) )}{\varphi^{\prime}(t)} + t \scal{\nabla f_{2}( \varphi(t) )}{\varphi^{\prime}(t)} = 0. \label{zerograd}\ea
Moreover, for any differentiable function $f$ with suitable domains,
\ba \dsp \pare{f\circ \varphi}^{\prime}(t) = \diff\pare{f\circ \varphi}_{t}(1) = \diff f_{\varphi(t)}\pare{\diff\varphi_{t}(1)} = \scal{\nabla f( \varphi(t) )}{\varphi^{\prime}(t)} .\label{composition}\ea
Inserting this in \eqref{zerograd} shows 
$ \dsp (1-t)\pare{f_{1}\circ \varphi}^{\prime}(t)  + t \pare{f_{2}\circ \varphi}^{\prime}(t) = 0$,
which is the same as:
\ba \dsp (1-t)u^{\prime}(t)  + t v^{\prime}(t) = 0, \text{ for any $t\in\croc{0,1}$}.\label{derivee}\ea
Since $\dsp\lim_{t\to 0} \frac{ (f_{1}\circ \varphi)^{\prime} (t) }{t}$ exists, \eqref{derivee} implies that:
 
\ba   v^{\prime}(t) = \pare{1-\frac1t} u^{\prime}(t), \text{ for any $t\in\croc{0,1}$}.  \label{equationsecondobj}\ea
By deriving \eqref{equationsecondobj} and multiplying by $u^{\prime}(t)$ in a suitable way, we obtain
\ba   u^{\prime}(t)v^{\prime\prime}(t) = \frac{1}{t^{2}} u^{\prime}(t)^{2} + \pare{1-\frac1t} u^{\prime}(t)u^{\prime\prime}(t) , \text{ for any $t\in\croc{0,1}$}.  \label{derivsecond}\ea
Using \eqref{equationsecondobj} in \eqref{derivsecond} gives $ u^{\prime}(t)v^{\prime\prime}(t) = \frac{1}{t^{2}} u^{\prime}(t)^{2} + v^{\prime}(t)u^{\prime\prime}(t).$
Thanks to the assertions on $f_{1}\circ \varphi$, we have that
$ u^{\prime}(t) v^{\prime\prime}(t) - u^{\prime\prime}(t) v^{\prime}(t) > \frac{1}{t^{2}} u^{\prime}(t)^{2} > 0, \text{ for any $t\in\croc{0,1}$}.$
 Thus, the Pareto front is a \textbf{convex} curve.\\
 Evaluating \eqref{derivee} at $t = 0$ and at $t = 1$ implies that $u^{\prime}(0) = 0, v^{\prime}(1) = 0.$
 And if we divide \eqref{derivee} by $t$ (resp. $1-t$) and take the limit to $0$ (resp. $1$), it follows that
$v^{\prime}(0) \neq 0$ (resp.\ $u^{\prime}(1) \neq 0$).
Thereby we also obtain the derivative assumptions on the extremal points.
 \qed
\end{proof} 
\begin{remark}
Note that the above result about the tangents in the extremal points have additional consequences: according to \cite{auger2009theory}, the assumptions of Theorem~\ref{bigtheo} imply that the extremal points are never included in any optimal $\mu$-distributions of the Hypervolume indicator.
\end{remark}
We now deduce the convexity of the Pareto front for convex-quadratic bi-objective problems and characterize the derivatives at the extremes of the front.
\begin{corollary}
For the problem $\mathcal{P}$, the Pareto front is a \textbf{convex} curve, with \textbf{vertical} tangent at $\pare{0, f_{2}(x_{1})}$ and \textbf{horizontal} tangent at $\pare{f_{1}(x_{2}),0}$. 
\end{corollary}
\begin{proof}
We will show that $ f_{1}\circ\varphi$ verifies the assumptions of Theorem~\ref{bigtheo}. From~\eqref{composition} we know that
\ba \dsp \pare{f_{1}\circ \varphi}^{\prime}(t) =  \scal{\nabla f_{1}( \varphi(t) )}{\varphi^{\prime}(t)} .\label{composition1} \ea
In addition, $ \nabla f_{1}( \varphi(t) ) = \frac{2}{\alpha} Q_{1}\pare{ \varphi(t)-x_{1} } $
and Eq.~\eqref{derivephi} of Proposition~\ref{paretoset} gives
$ t\croc{ (1-t) Q_{1} + t Q_{2} } \varphi^{\prime}(t) = Q_{1}\pare{ \varphi(t)-x_{1} }.$
Multiplying \eqref{composition1} by $t\in \croc{0,1}$ shows
\ba
 \dsp t \pare{f_{1}\circ \varphi}^{\prime}(t) = \frac{2}{\alpha} \scal{ \croc{ (1-t) Q_{1} + t Q_{2} }^{-1}Q_{1}\pare{ \varphi(t)-x_{1}} }{ Q_{1}\pare{ \varphi(t)-x_{1} } }.
\label{positivity}
\ea
Since $\croc{ (1-t) Q_{1} + t Q_{2} }^{-1}$ is a positive definite matrix, then 
$t \pare{f_{1}\circ \varphi}^{\prime}(t) \geq 0 . $
Let us prove that $\varphi(t) \neq x_{1}$, for $t \in (0,1]$. By contradiction, assume that there exists $t \in (0,1]$ such that $\varphi(t) = x_{1}$.
Then Equation~\eqref{eq:ParetoSet} in Proposition~\ref{paretoset} shows that:
$\dsp tQ_{2}\pare{ x_{2} - \varphi(t) } = (1-t)Q_{1}\pare{ \varphi(t)-x_{1} } = 0$,
which implies that $x_{2} = \varphi(t) = x_{1} $: that is impossible since $x_{1}\neq x_{2}$.
Hence, by \textit{reductio ad absurdum}, $\varphi(t) \neq x_{1}$, for $t \in (0,1].$
From~\eqref{positivity}, it follows that
\ba\pare{f_{1}\circ \varphi}^{\prime}(t) > 0 , \text{ for any $t \in (0,1].$} \label{firstpart}\ea
If we use again the relation from Proposition~\ref{paretoset}, we obtain
$\dsp \lim_{t\to 0} \frac{Q_{1}\pare{ \varphi(t)-x_{1} }}{t} = Q_{2}\pare{ x_{2} - \varphi(0) } = Q_{2}\pare{ x_{2} - x_{1} }.$
Injecting this result in~\eqref{positivity}, it follows that:
\ba \dsp\lim_{t\to 0} \frac{ (f_{1}\circ \varphi)^{\prime} (t) }{t} &=& 
\frac{2}{\alpha} \scal{ Q_{1}^{-1}Q_{2}\pare{ x_{2}-x_{1}} }{ Q_{2}\pare{ x_{2}-x_{1} }} \nonumber\\
		 &>& 0, \text{ since ($Q_{1}^{-1}$ is a positive definite matrix) }
\label{secondpart2}\ea
In the same way as above, we obtain that
\ba \dsp\lim_{t\to 1} \frac{ (f_{2}\circ \varphi)^{\prime} (t) }{1-t} = 
-\frac{2}{\beta} \scal{ Q_{2}^{-1}Q_{1}\pare{ x_{1}-x_{2}} }{ Q_{1}\pare{ x_{1}-x_{2} }} < 0 \enspace.
\label{secondobjpart}\ea
Equations \eqref{firstpart}, \eqref{secondpart2}, and \eqref{secondobjpart} allow us to apply Theorem~\ref{bigtheo}.
\qed
\end{proof}
We illustrate the previous corollary by taking three 
random instances of our general problem $\mathcal{P}$, with the scalings always chosen as
$\dsp \alpha = \beta = \max\pare{f_{1}(x_{2}), f_{2}(x_{1}) }$.
The Pareto fronts are presented in Figure~\ref{generalfront}. We observe that the Pareto fronts are convex and their derivatives are infinite on the left and zero on the right.
\begin{figure*}[h!]
        \includegraphics[width=0.52\textwidth]{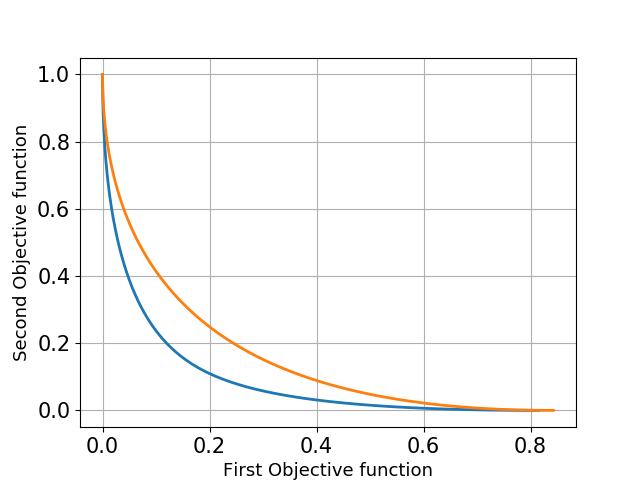}
 \hspace{-0.5cm}
        \includegraphics[width=0.52\textwidth]{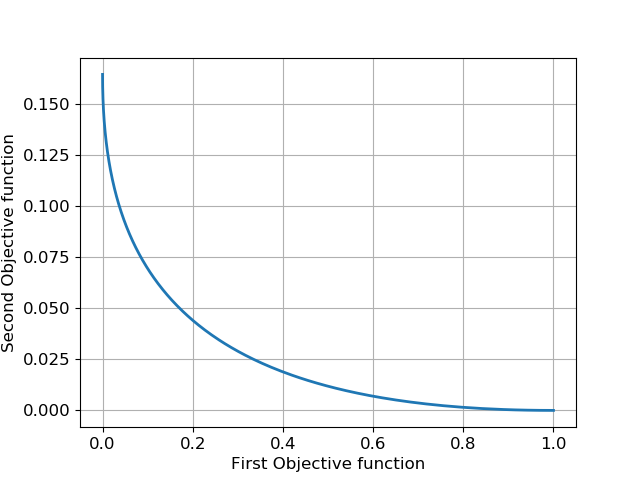}
    \caption{\label{generalfront} Left: Two Pareto fronts for $n=2$ represented in $\R^2$ with $Q_1, Q_2$ randomly sampled and different. Right: Pareto front for $n=10$ with matrices given in \eqref{eq:Qdim10}. }
\end{figure*}

\section{New Classes of  bi-objective test functions}
\label{classsection}

Bi-objective problems using convex-quadratic functions have been used to test MO algorithms (see for example~\cite{ihr2007a}). Problems where both Hessian matrices have the same eigenvalues have been used in particular.
Yet, test problems considered so far do not explore the full possibilities of properties that can be tested. We therefore extend the test problems from the literature to be able to capture more properties. To do so we present seven classes of bi-objective convex-quadratic problems where the eigenspectra of both Hessian matrices are equal.
A natural extension of these classes is to use in each objective different eigenspectra, $\Delta$, which leads in general to a nonlinear Pareto set.

The proposed construction parametrizes, apart from search space translations, \emph{all} bi-objective convex-quadratic functions with identical Hessian eigenspectrum in seven classes with increasing difficulty. The particular focus is on problems with a linear Pareto set in five of the seven classes.
Some classes represent essentially different problems, hence we do not expect uniform performance over all problems within each class.
Independently of the given construction, invariance to search space rotation can be tested by applying an orthogonal transformation to the input argument.

We start from a diagonal matrix $\Delta$ with positive entries that define a separable convex-quadratic function $f(x)= \frac{1}{\alpha} x^\top \Delta x$. For instance, $\Delta$ can be equal to the identity and we recover the sphere function.
If $ \Delta(1,1) = 1, \, \Delta(n,n) = 10^8$ and $\Delta(i,i)=10^4$, we recover the separable cig-tab function and if $\Delta(i,i) = 10^{6 \frac{i-1}{n-1}}$, we recover the separable ellipsoid function.

In the sequel, $O$ and $O_2$ denote orthogonal matrices. $O_{1}$ is either a permutation matrix, or an orthogonal matrix, depending on the context. The classes of problems proposed are  summarized in Table~\ref{tab1} and Table~\ref{tab2}.

\subsubsection*{The Sep problem classes}
We define the \textbf{Sep-$k$} class by considering two separable functions and place the optimum of $f_1$ in $0$ and of $f_2$ in the $k^{\rm th}$ unit vector: $f_{1,\Delta}^{ \text{\textbf{sep-$k$}}}(x) =  \frac{1}{\alpha} \pare{x-x_{1}}^\top \Delta \pare{x-x_{1}}$ and $f_{2,\Delta}^{ \text{\textbf{sep-$k$}}}(x) = \frac{1}{\beta} \pare{x-x_{2}}^\top \Delta \pare{x-x_{2}}$, where  $x_{1} = (0,\dots,0)^\top$ and $x_{2} = (0,\dots,0,\sqrt{n},0,\dots,0)^\top $ where $\sqrt{n}$ is at coordinate $k$. According to Corollary~\ref{segmentset}, the Pareto set of this class of problems is the line segment between the optima of the single-objective problems.
These problems allow to test the performance on separable problems with a Pareto set aligned with the coordinate axis and check the sensibility with respect to different axes (by varying $k$).

For the \textbf{Sep-O} class, we only change the location of the optimum of the second objective by taking $x_2 = O(1,\dots,1)^\top$. If $O$ has elements $\not\in\{-1,0,1\}$, the Pareto set is not anymore aligned with the coordinate system, but the objectives $f_1$ and $f_2$ themselves remain separable.
Comparing with class \textbf{Sep-$k$}, we can test whether having the Pareto set not aligned with the coordinate axis has an influence on the performance of the algorithm.

For the \textbf{Sep-Two-O} class, we define $f_{1,\Delta}^{ \text{\textbf{sep-Two-O}}}(x) = \frac{1}{\alpha} \pare{x-x_{1}}^\top \Delta \pare{x-x_{1}}$ and $f_{2,\Delta}^{ \text{\textbf{sep-Two-O}}}(x) = \frac{1}{\beta} \pare{x-x_{2}}^\top O_{1}^\top \Delta O_{1} \pare{x-x_{2}}$ where $O_{1}$ is a permutation matrix, $x_{1} = (0,\dots,0)^\top$ and $x_{2} = O(1,\dots,1)^\top$. The matrix $O_{1}^\top \Delta O_{1}$ is also diagonal, and thereby each function is separable. Yet the Pareto set is generally not a line segment anymore since we have different Hessian matrices. We can test here the difficulty of having a nonlinear Pareto set on separable functions.

\subsubsection*{The \textbf{One} and the \textbf{One-O} problem classes.}
We now consider non-separable problems with a line segment as Pareto set. We define
 $ f_{1,\Delta}^{\text{\textbf{one}}}(x) =  \frac{1}{\alpha} \pare{ x-x_{1}}^\top O_{1}^\top$ $\Delta O_{1} \pare{x-x_{1}} $ and $
  f_{2,\Delta}^{\text{\textbf{one}}}(x) = \frac{1}{\alpha} \pare{x-x_{2}}^\top O_{1}^\top \Delta O_{1}\pare{x-x_{2}}$,
where $O_{1}$ is an orthogonal matrix, $x_{1} = (0,\dots,0)^\top$ and $x_{2} = (1,\dots,1)^\top$. We replace $x_{2}$ by $O x_{2}$ to obtain the \textbf{One-O} problems.

These two problem classes allow to test the performance on non-separable problems that have a line segment as Pareto set comparing in particular to class \textbf{Sep-O}.
Up to a reformulation, the problems ELLI1 and CIGTAB1 from \cite{ihr2007a}  are from the \textbf{One-O} problem class.
Generally, we do not expect different performance over all problems of the \textbf{One} vs the \textbf{One-O} class.

\subsubsection*{The \textbf{Two} and the  \textbf{Two-O} problem classes.}
For these classes, we rotate each function independently; then the Pareto set is generally not a line segment anymore. We define
$ f_{1,\Delta}^{\text{\textbf{two}}}(x) =  \frac{1}{\alpha} \pare{ x-x_{1}}^\top O_{1}^\top\Delta O_{1} \pare{x-x_{1}} $ and $  f_{2,\Delta}^{\text{\textbf{two}}}(x) = \frac{1}{\alpha} \pare{x-x_{2}}^\top O_{2}^\top \Delta$ $O_{2}\pare{x-x_{2}}$, with $O_{1}$ orthogoanal, $x_{1} = (0,\dots,0)^\top$ and $x_{2} = (1,\dots,1)^\top$. The corresponding \textbf{O} problems are obtained with $O x_{2}$ replacing $x_{2}$.
All presented classes are subsets of the \textbf{Two-O} class.
ELLI2 and CIGTAB2 from \del{the paper~}\cite{ihr2007a} fall within the \textbf{Two-O} class.
Compared to the respective \textbf{One} classes, we can test the impact of having a nonlinear Pareto set.

\newcolumntype{R}[1]{>{\raggedleft\arraybackslash }b{#1}}
\newcolumntype{L}[1]{>{\raggedright\arraybackslash }b{#1}}
\newcolumntype{C}[1]{>{\centering\arraybackslash }b{#1}}
\pagestyle{empty}

\setlength{\tabcolsep}{0mm}
\newcolumntype{M}{>{\centering\arraybackslash}m{\dimexpr.30\linewidth-2\tabcolsep}}
\newcolumntype{L}{>{\centering\arraybackslash}m{\dimexpr.1\linewidth-2\tabcolsep}}

\begin{table}

\centering
\caption{ Unconstrained quadratic bi-objective test problems: $\dsp\Delta$ is a positive diagonal matrix, $O$ is an orthogonal matrix, $O_{1}$ is a permutation matrix.
}\label{tab1}
\begin{tabular}{  LMMM }
		\toprule
		      & \textbf{Sep-$k$} & \textbf{Sep-O} &  \textbf{Sep-Two-O} \\
		\midrule
		$x_1$ & $(0,\dots,0)^\top$ & $(0,\dots,0)^\top$ & $(0,\dots,0)^\top$ \\
		$x_2$ & $\underbrace{\pare{0,..,\sqrt{n},..,0}^\top }_{ \text{$\sqrt{n}$ is at row $k$} }$ &  $ O(1,\dots,1)^\top$ & $ O(1,\dots,1)^\top$  \\
		$Q_1, Q_2$ & 
		  $\Delta$, $\Delta$ & 
		  $\Delta$, $\Delta$ & 
		  $\Delta$, $O_{1}^\top \Delta O_{1}$ \\

		\begin{sideways}{Level sets}\end{sideways} & \includegraphics[scale=0.25]{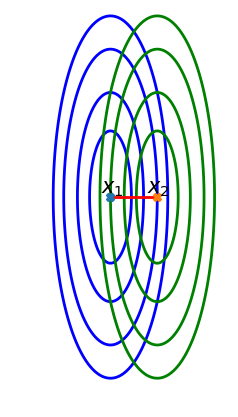} & \includegraphics[scale=0.25]{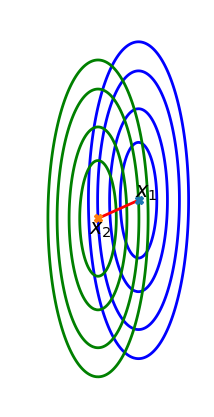} &  \includegraphics[scale=0.25]{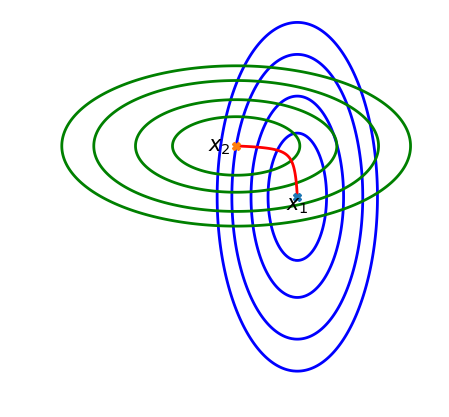}\\
		\bottomrule

\end{tabular}
\end{table}

\setlength{\tabcolsep}{0mm}
\newcolumntype{M}{>{\centering\arraybackslash}m{\dimexpr.224\linewidth-2\tabcolsep}}
\newcolumntype{L}{>{\centering\arraybackslash}m{\dimexpr.1\linewidth-2\tabcolsep}}

\begin{table}
	\centering
	\caption{ Unconstrained quadratic bi-objective test problems: $\dsp\Delta$ is a positive diagonal matrix, $O$, $O_{1}$ and $O_{2}$ are three independent orthogonal matrices.	
		}\label{tab2}
  \begin{tabular}{  LMMMM }
		\toprule
		      & \textbf{One} & \textbf{One-O} &  \textbf{Two} &  \textbf{Two-O} \\
		\midrule
		$x_1$ & $(0,\dots,0)^\top$ & $(0,\dots,0)^\top$ & $(0,\dots,0)^\top$ & $(0,\dots,0)^\top$ \\
		$x_2$ & $ (1,\dots,1)^\top$  & $ O(1,\dots,1)^\top$ & $ (1,\dots,1)^\top$  & $ O(1,\dots,1)^\top$  \\
		$Q_1, Q_2$ & 
		  $O_1^\top \Delta O_1$, $O_1^\top \Delta O_1$& $O_1^\top \Delta O_1$, $O_1^\top \Delta O_1$ & $O_1^\top \Delta O_1$, $O_2^\top \Delta O_2$ & $O_1^\top \Delta O_1$, $O_2^\top \Delta O_2$ \\
		\begin{sideways}{Level sets}\end{sideways} & 
		\hspace*{-1em}\includegraphics[scale=0.2]{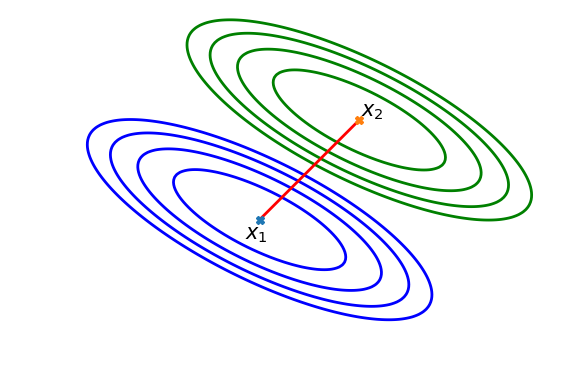} & \includegraphics[scale=0.2]{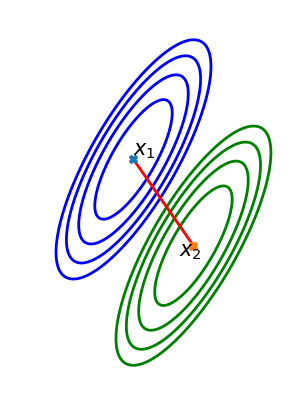} &  \includegraphics[scale=0.2]{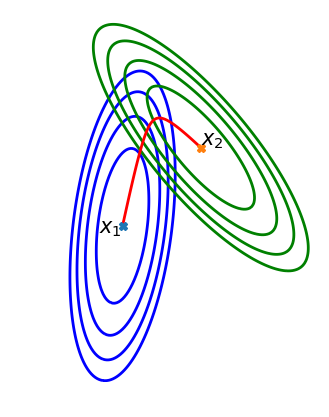}&  \includegraphics[scale=0.2]{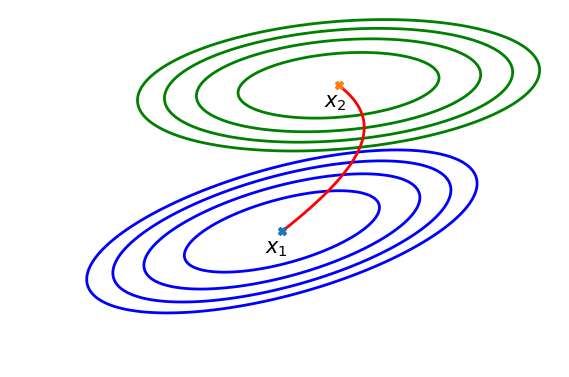}\\
		\bottomrule
\end{tabular}
\end{table}

\section{Summary}
We have presented an analytic description of the Pareto set for quadratic bi-objective problems.
We have shown that the Pareto set is a line segment when both objectives have proportional Hessian matrices and deduced a complete description of the Pareto front in that case.
We have also proven that some properties of the double-sphere are conserved in a wider framework that includes the general quadratic bi-objective problem:
the Pareto front remains convex and its vertical and horizontal tangents remain at the extremal points of the front.
Such assumptions on the derivatives imply that when looking at the optimal $\mu$-distributions of the Hypervolume indicator, the extremal points are always excluded \cite{auger2009theory}.
We have also presented several classes of problems, where each one tests a specific capability of the multiobjective algorithm.
\subsubsection*{Acknowledgments}
The Ph.D. of Cheikh Tour\'e is funded by Inria and Storengy. We particularly thank F. Huguet and A. Lange from Storengy for their strong support, practical ideas and expertise.

 \bibliographystyle{splncs04}
\bibliography{emo,allDimo}

\end{document}